\documentclass[leqno,12pt]{amsart}
\usepackage{stmaryrd} 
\usepackage{amssymb}
\theoremstyle{plain} 
\newtheorem{theorem}{\indent\sc Theorem}[section]
\newtheorem{lemma}[theorem]{\indent\sc Lemma}
\newtheorem{corollary}[theorem]{\indent\sc Corollary}

\theoremstyle{definition} 

\newtheorem{remark}[theorem]{\indent\sc Remark}

\begin{document}

\title[ Navier-Stokes-Cahn-Hilliard system]{On  the Cauchy problem of 3D incompressible Navier-Stokes-Cahn-Hilliard system  } 
\author[X. Zhao]{Xiaopeng Zhao } 


%

\address{
College of Sciences\endgraf
Northeastern University\endgraf
Shenyang 110819,~~~
P. R. China }
\email{zhaoxiaopeng@jiangnan.edu.cn}


\begin{abstract}
In this paper, we are concerned with the well-posedness and large time behavior of Cauchy problem for 3D incompressible Navier-Stokes-Cahn-Hilliard equations. First,  using Banach fixed point theorem, we establish the local well-posedness of solutions. Second, assuming $\|(u_0,\phi,\nabla\phi_0)\|_{\dot{H}^\frac12}$ is sufficiently small,  we obtain the  global well-posedness of solutions in the critical space. Moreover,  the optimal decay rates of the higher-order spatial derivatives of the solution are also obtained.
\end{abstract}\maketitle {\small
 \section{Introduction}

In Fluid Mechanics, the incompressible Navier-Stokes equations
 models the motions of single-phase fluids such as oil or water. Sometimes, we also need to understand the motion of binary fluid mixtures, that is fluids composed by either two phases of the same chemical species or phases of different composition. Diffuse interface methods  are widely used by many authors to describe the behavior of complex (e.g., binary) fluids \cite{DM,MD}. The Model H \cite{Gur,Ho,Morro} is a diffuse interface model for incompressible isothermal two-phase flows. This model include the incompressible Navier-Stokes equations for the (averged) velocity $u$ nonlinearly coupled with a convective Cahn-Hilliard equation for the (relative) concentration difference $\varphi$.

 Suppose that the temperature variations are negligible, taking the density is equal to $1$, and let the viscosity $\nu$  be constant, the Model H reduces to the incompressible Navier-Stokes-Cahn-Hilliard system \cite{2,Chella,JCP}
 \begin{equation} \label{1-0}
 \left\{ \begin{aligned}
         &\partial_tu-\nu\Delta u+u\cdot\nabla u+\nabla\pi=\mathcal{K}\mu\nabla\varphi, \\
                   &\hbox{div}u=0,\\
                   &\partial_t\varphi+u\cdot\nabla\varphi=\nabla\cdot(M\nabla\mu),\\
                   &\mu=-\varepsilon\Delta\varphi+\zeta F'(\varphi),
                           \end{aligned} \right.
                           \end{equation}
in $\Omega\times(0,\infty)$ with $\Omega\subseteq\mathbb{R}^3$ is a bounded domain or the whole space. In equations (\ref{1-0}),  $u$ denotes the mean velocity, $\pi$ represent the pressure and $\varphi$ is an order parameter related to the concentration of the two fluids (e.g. the concentration difference or the concentration of one component), respectively.      The quantities $\mu\geq0$, $M\geq0$ and $\mathcal{K}\geq0$  are positive constants that correspond to the kinematic viscosity of fluid, mobility constant and capillarity (stress)
coefficient, respectively. Moreover, $\mu$ denotes the chemical potential of the mixture which can be given by the variational derivative of the free energy functional
\begin{equation}\label{1-0x}
\mathcal{F}(\varphi)=\int_{\Omega}\left(\frac{\varepsilon}2|\nabla\varphi|^2+\zeta F(\varphi)\right)dx.
\end{equation}
Usually, we take   $F(\varphi)=\int_0^{\varphi}f(\xi)d\xi$ as   a double-well potential with two equal minima $\pm1$ corresponding to the two pure phases
\begin{equation}
\label{xy}
F(s)=\frac14(s^2-1)^2,
\end{equation}
or   a singular free energy
 \begin{equation}\label{xx}
F(s)=\frac{\theta}2[(1+s)\log(1+s)+(1-s)\log(1-s)]-\frac{\theta_c}2s^2,
\end{equation}
where $s\in[-1,1]$ and $0<\theta<\theta_c$. 
\begin{remark}In \cite{15G}, the authors pointed out that
$$
\mathcal{K}\mu\nabla\varphi=\mathcal{K}\nabla\left(\frac{\varepsilon}2|\nabla\varphi|^2+\xi F(\varphi)\right)
-\mathcal{K}\nabla\cdot(\nabla\varphi\otimes\nabla\varphi).
$$
Moreover,
the notation $\nabla \varphi\otimes\nabla \varphi$ denotes the $3\times3$ matrix whose $(i,j)$th entry is given by $\partial_i\varphi\cdot\partial_j\varphi$, $1\leq i,j\leq 3$, and hence
$$\nabla\cdot(\nabla \varphi\otimes\nabla \varphi)=\Delta\varphi\cdot\nabla\varphi+\frac12\sum_{k=1}^3\nabla|\nabla \varphi_k|^2.
$$
Therefore,  (\ref{1-0})$_1$ can be replaced by
$$
\partial_tu+u\cdot\nabla u-\nu\Delta u+\nabla\tilde{\pi}=-\mathcal{K}\Delta\varphi\cdot\nabla\varphi,
$$
with
$
\tilde{\pi}=\pi-\mathcal{K}\left(\frac{\varepsilon}2|\nabla\varphi|^2+\xi F(\varphi)\right)+ \frac{\mathcal{K}}2\sum_{k=1}^3\nabla|\nabla \varphi_k|^2.
$
\end{remark}
  In a sense, the incompressible Navier-Stokes-Cahn-Hilliard equation can be seen as  the coupling between the incompressible Navier-Stokes equations with the convective Cahn-Hilliard equations.
This coupling of equations   was highly studied from the theoretical and mathematical point of view. Till now, there is a large amount of literature on the mathematical analysis of initial-boundary value problem for Navier-Stokes-Cahn-Hilliard system in 2D or 3D case.   Abels et. al. \cite{ACMP,Ab,Abels}, Boyer \cite{7}, Eleuteri, Rocca and Schimperna \cite{Rocca}, Gal, Grasselli and Miranville \cite{Gal1}, Lam and Wu \cite{LaW}  studied the existence of global weak solution and unique strong solution for Navier-Stokes-Cahn-Hilliard  system with non-constant mobility; Gal and Grasselli \cite{15G} considered the asymptotic behavior of 2D Navier-Stokes-Navier-Stokes system, proved the existence of global attractor and exponential attractor; Colli, Frigeri and Grasselli \cite{PColli}, Frigeri, Grasselli and Krejci \cite{SFri}, Frigeri and Grasselli \cite{Frig} and Frigeri, Grasselli and Rocca \cite{Fri2} replaced the chemical potential (\ref{1-0})$_4$ by the following nonlocal model
$$
\mu=\int_{\Omega}J(x-y)dy\varphi-\int_{\Omega}J(x-y)\varphi(y)dy+\xi F'(\varphi),
$$
 obtained the nonlocal Navier-Stokes-Cahn-Hilliard equations with double-well potential or singular free energy, studied the well-posedness and long time behavior of solutions; Bosia and Gatti \cite{SB}, Cherfils and Madalina \cite{CPAA} and You, Li and Zhang \cite{YB} investigated the properties of solutions for Navier-Stokes-Cahn-Hilliard equations with dynamic boundary conditions, studied the well-posedness and long time behavior of solutions respectively; In \cite{ZF}, Zhou and Fan proved the vanishing viscosity limit of solutions for the initial boundary value problem of  2D  Navier-Stokes-Cahn-Hilliard equations; Giorgini,  Miranville and Temam \cite{AG} studied the Navier-Stokes-Cahn-Hilliard system with viscosity depending on concentration and logarithmic potential,  proved uniqueness of weak solutions in 2D, and existence and uniqueness of strong solutions (global in 2D and local in 3D) with divergence free initial velocity in $H_0^1$. Liu and Shen \cite{LS}, Key and Welfold \cite{Kay} and Feng \cite{Feng} studied the  numerical approximation of solutions,  etc.

There are only few results on the mathematical analysis of the Cauchy problem for incompressible Navie-Stokes-Cahn-Hilliard system. The first related result was obtained by Starovoitov \cite{STA}. The author supposed that $F$ is a suitably smooth double-well potential, studied the qualitative behavior  as $t\rightarrow\infty$ of the solutions for 2D Cauchy problem. 
The other paper on the Cauchy problem was written by Cao and Gal \cite{Cao}. In their paper, the authors established the global retularity and uniqueness of strong/classical solutions  for the Cauchy problem of 2D NS-CH equations with mixed partial viscosity and mobility. We remark that because of the difficulty caused by  the convective term (the order of $u\cdot\nabla\varphi$ in (\ref{1-0})$_3$ is low, one can't control it by a fourth order linear term directly), the coupling between $u$ and $\varphi$ (especially the term on the right hand side of (\ref{1-0})$_1$) and the chemical potential, there's no paper on the   Cauchy problem  of three-dimensional incompressible  Navier-Stokes-Cahn-Hilliard equations. Thus a natural question is how to study the properties of solutions for the Cauchy problem  of 3D Navier-Stokes-Cahn-Hilliard equations.

Since $\pm1$ can be seen as two   equal minima of the double well-potential. In this paper, let $\omega_0=\pm1$, we  consider the global existence and the time decay rate of solutions for the following Cauchy problem:
\begin{equation} \label{1-1c}
 \left\{ \begin{aligned}
         &\partial_tu- \Delta u+u\cdot\nabla u+\nabla\pi=-\Delta\phi\cdot\nabla\phi , \\
                   &\hbox{div}u=0,\\
                   &\partial_t\phi+u\cdot\nabla\phi+\Delta^2\phi=\Delta [(\phi+\omega_0)^3-\phi] ,
                   \\
                   &(u,\phi)|_{x=0}=(u_0(x),\phi_0(x)),
                           \end{aligned} \right.\quad (x,t)\in\mathbb{R}^3\times\mathbb{R}^+.
                           \end{equation}
                           where $\phi=\varphi-\omega_0$ and $\phi_0=\varphi_0-\omega_0$.
For convenience, we rewrite (\ref{1-1c}) as
 \begin{equation} \label{1-1}
 \left\{ \begin{aligned}
         &\partial_tu +u\cdot\nabla u+\nabla\pi-\Delta u=-\Delta\phi\cdot\nabla\phi , \\
                   &\hbox{div}u=0,\\
                   &\partial_t\phi+u\cdot\nabla\phi+\Delta^2\phi-  \Delta \phi=\Delta((\phi+\omega_0)^3-\phi)  , \\
                   &(u,\phi)|_{t=0}=(u_0(x), \phi _0(x)),
                           \end{aligned} \right.\quad\hbox{in}~ \mathbb{R}^3\times (0,\infty).
                           \end{equation}
\begin{remark}

We note that the  Laplacian operator $(-\Delta)^{\delta}$ ($\delta\in\mathbb{R}$) can be defined through the Fourier transform, namely
\begin{equation}\label{1-4c}
(-\Delta)^{\delta}f(x)=\Lambda^{2\delta}f(x)=\int_{\mathbb{R}^3}|x|^{2\delta}\hat{f}(\xi)e^{2\pi ix\cdot\xi}d\xi,
\end{equation}
where  $\widehat{f}$ is the Fourier transform of $f$. Moreover, we  use the notation $A\lesssim B$ to mean that $A\leq cB$ for a universal constant $c>0$ that only depends on the parameters coming from the problem and the indexes $N$ and $s$  coming from the regularity on the data. We also employ $C$ for positive constant depending additionally on the initial data.
\end{remark}
First of all, we consider the local well-posedness of solutions for problem (\ref{1-1}), i.e., we prove the following theorem:
\begin{theorem}[Local well-posedness]
\label{thm1.0}
Suppose that $(u_0,\phi_0)\in H^1(\mathbb{R}^3) $. Then, there exists a small time $\tilde{T}>0$ and a unique strong solution $(u,\phi)(x,t)$ to system (\ref{1-1}) satisfying
\begin{equation} \label{local-1}
\left\{ \begin{aligned}
        &u\in L^{\infty}([0,\tilde{T}];H^1)\bigcap L^2(0,\tilde{T};H^2), \\
               & \phi \in L^{\infty}(0,\tilde{T};H^1)\bigcap L^2(0,\tilde{T};H^3).
                          \end{aligned} \right.
                          \end{equation}
\end{theorem}

The second aim   of this paper is to consider the small data global well-posedness of solutions for system (\ref{1-1}).
\begin{theorem}
\label{thm1.1}
Let $N\geq1$, assume that $(u_0,  \phi_0)\in H^N(\mathbb{R}^3)\times H^{N+1}(\mathbb{R}^3)$, and there exists a constant $\delta_0>0$ such that if
\begin{equation}
\label{1-2}
\|u_0 \|_{\dot{H}^\frac12}+\| \phi_0 \|_{\dot{H}^\frac12}+\| \phi_0 \|_{\dot{H}^\frac32} \leq\delta_0 .
\end{equation}
  Then there exists a unique global solution $(u,\phi)$ satisfying that for all $T\geq0$,
\begin{equation} \label{local-1x}
\left\{ \begin{aligned}
        &u\in L^{\infty}( 0,T ;H^N),\quad\nabla u\in L^2(0,T;H^{N}), \\
               &  \phi \in L^{\infty}(0,T;H^{N+1}),\quad\nabla\phi,~\Delta\phi\in L^2(0,T;H^{N+1}).
                          \end{aligned} \right.
                          \end{equation}
\end{theorem}
\begin{remark}
Note that system (\ref{1-1}) can be seen as the coupling between the Navier-Stokes equations and convective Cahn-Hilliard equation. Simple calculation shows that $\dot{H}^{\frac12}(\mathbb{R}^3)$ is a critical space for 3D Navier-Stokes equations. Moreover, if $v(x,t)$ is a smooth solution to 
\begin{equation}\label{0-6}
v_t-\Delta^2v+\nabla\cdot(|\nabla v|^2\nabla
v)=0,
\end{equation}in $\mathbb{R}^2$, then for each $\lambda>0$,
 $$
v_{\lambda}(x,t)=\lambda^{ -1} v(\lambda x,\lambda^4t)
 $$
 also solves (\ref{0-6}) unless we consider the initial condition. And the following scaling identity
 $$
 \|\nabla^k v_{\lambda}(\cdot,t)\|_{L^p(\mathbb{R}^3)}=\lambda^{1+k-1-\frac 3p}\|\nabla^k v(\cdot,\lambda^4t)\|_{L^p(\mathbb{R}^3)}.
 $$
 Hence, $p$ is the critical exponent if it satisfies
\begin{equation}\label{0-7}
 \frac 3p=1+k-1=k.
\end{equation}
Let $k=2$, by (\ref{0-7}), we find that $p=\frac 32$, this implies that $\dot{H}^{\frac32}(\mathbb{R}^3)=\dot{W}^{\frac32,2}(\mathbb{R}^3)$ is a critical space of equation (\ref{0-6}).  Hence, Theorem \ref{thm1.1} may be seen as a small initial data global well-possedness result in the critical Sobolev space.

\end{remark}

In the end, we also study the temporary decay rate of strong solutions for system (\ref{1-1}).
\begin{theorem}\label{thm1.2}
Suppose that all the assumptions of Theorem \ref{thm1.1} hold. If further, $(u_0,  \nabla\phi_0)\in\dot{H}^{-s}(\mathbb{R}^3)$ for some $s\in[0,\frac12]$, then for all $t\geq0$,
\begin{equation}
\label{1-4}
\|\Lambda^{-s}u(t)\|_{L^2}^2  +\|\Lambda^{-s}  \phi(t)\|_{L^2}^2  +\|\Lambda^{-s} \nabla\phi(t)\|_{L^2}^2\leq C,
\end{equation}
and
\begin{equation}\label{1-5}
\|\Lambda^ku(t)\|_{L^2}+ \|\Lambda^{k} \phi(t)\|_{L^2}+ \|\Lambda^{k+1} \phi(t)\|_{L^2}\leq C(1+t)^{-\frac{l+1}2},\quad\hbox{for}~k=0,1,\cdots, N-1 .
\end{equation}\end{theorem}

Note that the Hardy-Littlewood-Sobolev theorem implies that for $p\in[\frac32,2]$,  $L^{p}(\mathbb{R}^3)\subset\dot{H}^{-s}(\mathbb{R}^3)$ with $s=3(\frac1p-\frac12)\in[0,\frac12]$. Then, on the basis of Theorem \ref{thm1.1}, we easily obtain the following corollary of the optimal decay estimates.
\begin{corollary}
\label{cor1.1}Under the assumptions of Theorem \ref{thm1.1}, if we replace the $\dot{H}^{-s}(\mathbb{R}^3)$ assumption by $(u_0,\phi_0,\nabla\phi_0)\in L^{p}(\mathbb{R}^3)$ $(\frac32\leq p\leq 2)$, then the following decay estimate holds:
\begin{equation}
\label{1-9}
\|\Lambda^ku(t)\|_{L^2}+ \|\Lambda^{k} \phi(t)\|_{L^2}+ \|\Lambda^{k+1} \phi(t)\|_{L^2}\leq C(1+t)^{-\sigma_k},\quad\hbox{for}~k=0,1,\cdots, N-1,
\end{equation}where $$\sigma_k=\frac32\left(\frac1p -\frac12\right)+\frac l2 .
$$
\end{corollary}
\begin{remark}Since the decay rate (\ref{1-9}) is equivalent to the decay rate of second order heat equation
 \begin{equation} \label{heat}
\left\{\begin{aligned}
      & u_t +\Delta  u=0, \quad x\in\mathbb{R}^3,~t\geq0,\\
       &u(x,0)=u_0(x),
                  \end{aligned}\right.
                          \end{equation}
                          Hence, it is optimal.
\end{remark}

The main difficulties to consider the Cauchy problem of 3D Navier-Stokes-Cahn-Hilliard equations are how  to deal with the convective term $u\cdot\nabla \phi$, the coupling between $u$ and $\phi$ and the  linear term of the double-well potential. Since the principle part of (\ref{1-1})$_3$ is a fourth-order linear term and the convective term   is only a first-order nonlinear term,  due to Sobolev's embedding $L^{\frac{6}{3-2s}}(\mathbb{R}^3)\subset \dot{H}^s(\mathbb{R}^3)$, we cannot control $\|\nabla^kf(\phi)\|_{L^p}$ through $\|\nabla^{k+2}\phi\|_{L^{2}}$. In order to overcome this difficulty, we borrow a second-order term from the double-well potential, rewrite (\ref{1-1})$_3$ as
     \begin{equation} \label{1-10}
    \partial_t\phi+\Delta^2\phi- \Delta \phi=-u\cdot\nabla \phi+\Delta\left[(\phi+\omega_0)^3- 2\phi\right].
\end{equation}
Hence, one can control the convective term by the last term of the left hand side of equation  (\ref{1-10}). Moreover, in order to overcome the difficulty caused by the coupling between $u$ and $\phi$, we   assume   the initial data $\|(u_0,\nabla\phi_0) \|_{\dot{H}^{\frac12} }$ is sufficiently small, obtain a priori estimates on $(u, \nabla\phi)$.
On the other hand,  we establish the suitable a priori estimates in the negative Sobolev space $\dot{H}^{-s}$ ($0\leq s\leq\frac12$),   obtain the  optimal decay rate of strong solutions for the Cauchy problem of Navier-Stokes-Cahn-Hilliard equation in the whole spaces $\mathbb{R}^3$.

The structure of this paper is organized as follows.  In Section 2, we introduce some preliminary results, which are useful to prove our main results. Section 3 is  devoted to prove the local well-posedness of solutions. In Section 4, we  prove the small initial data global well-posedness of solutions in critical Sobolev space. Section 5 is devoted to study the decay rate of solutions for system (\ref{1-1}).
\section{Preliminaries}
In this section, we introduce some helpful results in $\mathbb{R}^3$.

There is a useful Sobolev embedding theorem, which will be used in the proofs.
\begin{lemma}[\cite{Rob}]\label{x}
There exists a constant $c$ such that for $0\leq s<\frac32$,
$$
\|u\|_{L^{\frac6{3-2s}}}\leq c\|u\|_{\dot{H}^s}\quad\hbox{for~all}~~u\in \dot{H}^s(\mathbb{R}^3).
$$
\end{lemma}

The following Gagliardo-Nirenberg inequality was proved in \cite{Nirenberg}.
\begin{lemma}[\cite{Nirenberg}]
\label{lem2.1}
Let $0\leq m,\alpha\leq l$, then we have
\begin{equation}\label{2-1}
\|\Lambda^{\alpha}f\|_{L^p}\lesssim\|\Lambda^mf\|_{L^q}^{1-\theta}\|\Lambda^lf\|_{L^r}^{\theta},
\end{equation}
where $\theta\in[0,1]$ and $\alpha$ satisfies
\begin{equation}\label{2-2}
\frac{\alpha}3-\frac1p=\left(\frac m3-\frac1q\right)(1-\theta)+\left(\frac l3-\frac1r\right)\theta.
\end{equation}
Here, when $p=\infty$, we require that $0<\theta<1$.
\end{lemma}

The following  Kato-Ponce inequality  is of great importance in the proofs.
\begin{lemma}[\cite{KP}]\label{KP}
Let $1<p<\infty$, $s>0$. There exists a positive constant $C$ such that
\begin{equation}
\label{K-1}
\|\Lambda^s(fg)-f\Lambda^sg\|_{L^p}\leq C(\|\nabla f\|_{L^{p_1}}\|\Lambda^{s-1}g\|_{L^{p_2}}+\|\Lambda^sf\|_{L^{q_1}}\|g\|_{L^{q_2}}),
\end{equation}
and
\begin{equation}
\label{K-2}
\|\Lambda^s{fg}\|_{L^p}\leq C(\|f\|_{L^{p_1}}\|\Lambda^sg\|_{L^{p_2}}+\|\Lambda^sf\|_{L^{q_1}}\|g\|_{L^{q_2}},
\end{equation}
where $p_2,q_2\in(1,\infty)$ satisfying $\frac1p=\frac1{p_1}+\frac1{p_2}=\frac1{q_1}+\frac1{q_2}$.
\end{lemma}

We also introduce the Hardy-Littlewood-Sobolev theorem, which implies the following $L^p$ type inequality.
\begin{lemma}[\cite{Stein,15}]
\label{lem2.3}
Let $0\leq s<\frac32$, $1<p\leq 2$ and $\frac12+\frac s3=\frac1p$, then
\begin{equation}
\label{2-4}
\|f\|_{\dot{H}^{-s}}\lesssim\|f\|_{L^p}.
\end{equation}
\end{lemma}

The  special Sobolev interpolation lemma will be used in the proof of Theorem \ref{thm1.1}.
\begin{lemma}[\cite{W,Stein}]
\label{lem2.2}
Let $s,k\geq0$ and $l\geq0$, then
\begin{equation}
\label{2-3}
\|\Lambda^lf\|_{L^2}\leq\|\Lambda^{l+1}f\|_{L^2}^{ 1-\theta }\|f\|_{\dot{H}^{-s}}^{ \theta },\quad\hbox{with}~\theta=\frac1{l+1+s}.
\end{equation}
\end{lemma}
\section{Local well-posedness}
In this section, by using Banach fixed point theorem, we   prove the local well-posedness for system (\ref{1-1}). Let
$$
\mathcal{A}:=\{v\in C([0,T];H^1),~~~\|v\|_{L^{\infty} (0,T;H^1)} \leq R\},
$$
for some positive constant $R$ to be determined latter.

Assume that $(\tilde{u}, \tilde{\phi}) \in\mathcal{A}\times   \mathcal{A}$ be given   and $ (\tilde{u},\tilde{\phi}) (\cdot,0)= (u_0,\phi_0) $. Consider
\begin{equation} \label{eq-1}
 \left\{ \begin{aligned}
         &\partial_tu- \Delta u+\tilde{u}\cdot\nabla u+\nabla\pi=-\Delta\tilde{\phi}\cdot\nabla\phi , \\
                   &\hbox{div}u=0,\\
                   &\partial_t\phi+\tilde{u}\cdot\nabla\phi+\Delta^2\phi- \Delta\phi=\Delta[(\phi+\omega_0)((\tilde{\phi}+\omega_0)^2-2)] ,
                   \\
                   &(u,\phi)|_{x=0}=(u_0(x),\phi_0(x)).
                           \end{aligned} \right.\quad (x,t)\in\mathbb{R}^3\times\mathbb{R}^+.
                           \end{equation}
Let $(u,\phi)(x,t)$ be the unique strong solution to (\ref{eq-1}). Define the fixed point map $F: (\tilde{u},\tilde{\phi} )   \in\mathcal{A} \times \mathcal{A} \rightarrow (u,\phi )\in\mathcal{A} \times \mathcal{A} $. We will prove that the map $F$ maps $\mathcal{A} \times  \mathcal{A} $ into $\mathcal{A} \times  \mathcal{A}$ for suitable constant $R$ and small $T>0$ and $F$ is a contraction mapping on $\mathcal{A} \times \mathcal{A} $. Therefore, $F$ has a unique fixed point in $\mathcal{A}\times  \mathcal{A} $. This proves the result.

\begin{lemma}
\label{lema-1}
Let $(\tilde{u},\tilde{\phi} )   \in\mathcal{A} \times \mathcal{A}  $ be given  and $ (\tilde{u},\tilde{\phi}) (\cdot,0)= (u_0,\phi_0) $. Assume that the  constant $ C_0>0$  is independent of $R$. Then, there exists a unique strong solution $u(x,t)$ for system (\ref{eq-1}) such that
\begin{equation}\label{eq-3}
\|(u,\phi)\|_{L^{\infty}(0,T;H^1)} +\|(u,\phi)\|_{L^2(0,T;H^2)}  +\|\phi\|_{L^2(0,T;H^3)}  \leq \tilde{C}_0,
\end{equation}
for some small $T>0$.
\end{lemma}
\begin{proof}
Since system (\ref{eq-1}) is   linear with regular $\tilde{u}$ and $\tilde{\phi}$ whose existence and uniqueness   can be found in Temem \cite{Temam}, then we only need to prove the a priori estimates (\ref{eq-3}) in the following.

Multiplying (\ref{eq-1})$_1$  by $u$,  multiplying (\ref{eq-1})$_2$ by $\phi$, integrating by parts over $\mathbb{R}^3$, summing them up, by using the condition of divergence free of $u$, we derive that
\begin{equation}\begin{aligned}
\label{eq-4}&
\frac12\frac d{dt}(\|u\|_{L^2}^2+\|\phi\|_{L^2}^2 )+\|\nabla u\|_{L^2}^2+ \|\Delta\phi\|_{L^2}^2+\kappa\|\nabla\phi\|_{L^2}^2
\\
=& -\int_{\mathbb{R}^3 }\Delta\tilde{\phi}\cdot\nabla\phi  udx
-\int_{\mathbb{R}^3}\tilde{u}\cdot\nabla\phi\cdot\phi dx
+\int_{\mathbb{R}^3}[(\phi+\omega_0)((\tilde{\phi}+\omega_0)^2-2)]\Delta\phi dx\\
=&J_1+J_2+J_3 .
\end{aligned}\end{equation}
Note that
\begin{equation}
\label{eq-5}\begin{aligned}
J_1\lesssim&\|u\|_{L^6}\|\Delta\tilde{\phi}\|_{L^2}\|\nabla\phi\|_{L^3}
\lesssim \frac12\|\nabla u\|_{L^2}^2+\|\Delta\tilde{\phi}\|_{L^2}^2\|\nabla \phi \|_{L^2}\|\Delta \phi \|_{L^2}
\\\lesssim& \frac12\|\nabla u\|_{L^2}^2+\frac{1 }8\|\Delta\phi\|_{L^2}^2+R^4\|\nabla \phi \|_{L^2}^2\\
\lesssim& \frac12\|\nabla u\|_{L^2}^2+\frac{1 }8\|\Delta\phi\|_{L^2}^2+\frac{1 }8\|\Delta\phi\|_{L^2}^2+R^8\| \phi \|_{L^2}^2
.
\end{aligned}
\end{equation}
\begin{equation}\label{eq-6}
J_2\lesssim\|\tilde{u}\|_{L^3}\|\nabla\phi\|_{L^2}\|\phi\|_{L^6}\lesssim R\|\nabla\phi\|^2_{L^2}\lesssim\frac{1 }8\|\Delta\phi\|_{L^2}^2+R^2\| \phi \|_{L^2}^2.
\end{equation}
\begin{equation}
\label{eq-7}\begin{aligned}
J_3\lesssim&\|\Delta\phi\|_{L^2}\|\phi+\omega_0\|_{L^6}\|(\tilde{\phi}+\omega_0)^2-2\|_{L^3}
\\
\lesssim&\|\Delta\phi\|_{L^2}\|\phi+\omega_0\|_{L^6}\| (\tilde{\phi}+\omega_0)-\sqrt{2}\|_{L^6}\|(\tilde{\phi}+\omega_0)+\sqrt{2}\|_{L^6}
\\
\lesssim&\frac{1 }{16}\|\Delta\phi\|^2_{L^2}+\|\nabla\phi\|_{L^2}^2\|\nabla\tilde{\phi}\|_{L^2}^4
\\
\lesssim&\frac{1 }{16}\|\Delta\phi\|^2_{L^2}+R^4\|\nabla\phi\|_{L^2}^2
\\
\lesssim&\frac{1 }{16}\|\Delta\phi\|^2_{L^2}+\frac{1 }{16}\|\Delta\phi\|^2_{L^2}+R^8\|\phi\|_{L^2}^2.
\end{aligned}\end{equation}
Combining (\ref{eq-4})-(\ref{eq-7}) together gives
\begin{equation}\begin{aligned}
\label{eq-10}
 \frac d{dt}(\|u\|_{L^2}^2+\|\phi\|_{L^2}^2 )+\|\nabla u\|^2_{L^2}+\|\Delta\phi\|_{L^2}^2+\kappa\|\nabla\phi\|_{L^2}^2
 \lesssim (R^2+R^8)(\|u\|_{L^2}^2+\|\phi\|_{L^2}^2 )
,
\end{aligned}\end{equation}
which yields that
\begin{equation}
\label{eq-11}
\|u\|_{L^2}^2+\|\phi\|_{L^2}^2 +\int_0^T(\|\nabla u\|_{L^2}^2+  \|\Delta\phi\|_{L^2}^2+ \|\nabla\phi\|_{L^2}^2)ds\leq C_1,
\end{equation}
provided that $(R^2+ R^8 )T<1$.

Multiplying (\ref{eq-1})$_1$  by $\Delta u$,  multiplying (\ref{eq-1})$_2$ by $\Delta\phi$, integrating by parts over $\mathbb{R}^3$, summing them up, by using the condition of divergence free of $u$, we derive that
\begin{equation}\begin{aligned}
\label{eq-12}&
\frac12\frac d{dt}(\|\nabla u\|_{L^2}^2+\|\nabla\phi\|_{L^2}^2 )+\|\Delta u\|_{L^2}^2+ \|\nabla\Delta\phi\|_{L^2}^2+\kappa\|\Delta\phi\|_{L^2}^2
\\
=&-\int_{\mathbb{R}^3}\tilde{u}\cdot \nabla u\cdot\Delta udx -\int_{\mathbb{R}^3}\Delta\tilde{\phi}\cdot\nabla \phi\cdot  \Delta udx
+\int_{\mathbb{R}^3}\tilde{u}\cdot\nabla\phi\cdot\Delta\phi dx\\&
-\int_{\mathbb{R}^3}\Delta[(\phi+\omega_0)((\tilde{\phi}+\omega_0)^2-2)]\Delta\phi dx
\\
=&J_4+J_5+J_6+J_7  .
\end{aligned}\end{equation}
Note that
\begin{equation}\begin{aligned}
\label{eq-13}
J_4\lesssim&\|\Delta u\|_{L^2}\|\tilde{u}\|_{L^6}\|\nabla u\|_{L^3}\lesssim\|\nabla \tilde{u}\|_{L^2}\|\nabla u\|^{\frac12}_{L^2}\|\Delta u\|_{L^2}^{\frac32}
\\
\lesssim&\frac14\|\Delta u\|_{L^2}^2+\|\nabla \tilde{u}\|_{L^2}^4\|\nabla u\|^{2}_{L^2}
\\
\lesssim&\frac14\|\Delta u\|_{L^2}^2+R^4\|\nabla u\|^{2}_{L^2}.
\end{aligned}\end{equation}
\begin{equation}
\begin{aligned}
\label{eq-14}
J_5\lesssim&\|\Delta u\|_{L^2}\|\Delta\tilde{\phi}\|_{L^2}\|\nabla\phi\|_{L^{\infty}}
\\\lesssim &\|\Delta u\|_{L^2}\|\Delta\tilde{\phi}\|_{L^2}\|\nabla\phi\|_{L^2}^{\frac14}\|\nabla\Delta\phi\|_{L^2}^{\frac34}
\\
\lesssim&\frac14\|\Delta u\|_{L^2}^2+\|\Delta\tilde{\phi}\|^2_{L^2}\|\nabla\phi\|^{\frac12}_{L^2}\|\nabla\Delta\phi\|_{L^2}^{\frac32}
\\
\lesssim& \frac14\|\Delta u\|_{L^2}^2+\frac{1 }4\|\nabla\Delta\phi\|^2_{L^2}+\|\Delta\tilde{\phi}\|_{L^2}^8\|\nabla\phi\|_{L^2}^2
\\
\lesssim&\frac14\|\Delta u\|_{L^2}^2+\frac{1 }4\|\nabla\Delta\phi\|^2_{L^2}+R^8\|\nabla\phi\|_{L^2}^2.
\end{aligned}\end{equation}
\begin{equation}\label{eq-8}\begin{aligned}
J_6\lesssim&\|\Delta\phi\|_{L^6}\|\tilde{u}\|_{L^3}\|\nabla\phi\|_{L^2}\lesssim\frac14\|\nabla\Delta \phi\|_{L^2}^2+\|\tilde{u}\|_{L^3}^2\|\nabla\phi\|_{L^2}^2
\\\lesssim&\frac14\|\nabla\Delta \phi\|_{L^2}^2+R^2\|\nabla\phi\|_{L^2}^2.\end{aligned}
\end{equation}
\begin{equation}
\begin{aligned}\label{eq-9}
J_7\lesssim&\|\nabla\Delta\phi\|_{L^2}\|\nabla[(\phi+\omega_0)((\tilde{\phi}+\omega_0)^2-2)]\|_{L^2}
\\
\lesssim&\|\nabla\Delta\phi\|_{L^2}(\|\nabla\phi\|_{L^2}\|(\tilde{\phi}+\omega_0)^2-2\|_{L^{\infty}}+\|\phi\|_{L^6}\|\nabla((\tilde{\phi}+\omega_0)^2-2)\|_{L^3})
\\
\lesssim&\|\nabla\Delta\phi\|_{L^2}\|\nabla\phi\|_{L^2}\|\nabla\tilde{\phi}\|_{L^2}\|\Delta\tilde{\phi}\|_{L^2}
\\
\lesssim&\frac14\|\nabla\Delta\phi\|_{L^2}^2+\|\nabla\phi\|_{L^2}^2\|\nabla\tilde{\phi}\|_{L^2}^2\|\Delta\tilde{\phi}\|_{L^2}^2
\\
\lesssim&\frac14\|\nabla\Delta\phi\|_{L^2}^2+R^4\|\nabla\phi\|_{L^2}^2.
\end{aligned}\end{equation}
Combining (\ref{eq-12})-(\ref{eq-9}) together gives
\begin{equation}\begin{aligned}
\label{eq-15}&
 \frac d{dt}(\|\nabla u\|_{L^2}^2+\|\nabla\phi\|_{L^2}^2 )+\|\Delta u\|_{L^2}^2+ \|\nabla\Delta\phi\|_{L^2}^2+\kappa\|\Delta\phi\|_{L^2}^2\\&\lesssim(R^2+R^4+R^8)(\|\nabla u\|_{L^2}^2+\|\nabla\phi\|_{L^2}^2 ),
\end{aligned}\end{equation}
which yields that
\begin{equation}
\label{eq-16}
\|\nabla u\|_{L^2}^2+\|\nabla\phi\|_{L^2}^2 +\int_0^T(\|\Delta u\|_{L^2}^2+  \|\nabla\Delta\phi\|_{L^2}^2+ \|\Delta\phi\|_{L^2}^2)ds\leq C_1,
\end{equation}
provided that $(R^2+ R^4+R^8 )T<1$. This complete the proof.

\end{proof}

Due to Lemma  \ref{lema-1}, we can take $R=2\sqrt{C_1}$, and thus, $F$ maps $\mathcal{A} \times \mathcal{A} $ into $\mathcal{A} \times \mathcal{A} $. We prove that $F$ is a contraction mapping in the sense of weaker norm in the following.
\begin{lemma}
\label{lem-3}
There exists a constant $\delta\in(0,1)$ such that for any $( \tilde{u}_i,\tilde{\phi}_i) $ $(i=1,2)$,
\begin{equation}
\label{eq-17}\begin{aligned}&
\|F(\tilde{u}_1,
\tilde{\phi}_1 )-F(\tilde{u}_2,
\tilde{\phi}_2 )\|_{L^2(0,T;H^1)}
\leq \delta\left(
\|\tilde{u}_1-\tilde{u}_2\|_{L^2(0,T;H^1)}+
\|\tilde{\phi}_1-\tilde{\phi}_2\|_{L^2(0,T;H^1)}\right)  ,\end{aligned}
\end{equation}
for some small $T>0$.
\end{lemma}
\begin{proof}
Suppose that $ (u_i,\phi_i)(x,t)$ $(i=1,2)$ are the solutions to problem (\ref{eq-1})  corresponding to $ ( \tilde{u}_i,\tilde{b}_i) $.
Denote
$$
u=u_1-u_2,\quad~~\tilde{u}=\tilde{u}_1-\tilde{u}_2,\quad~~\phi=\phi_1-\phi_2,\quad~~\tilde{\phi}=\tilde{\phi}_1-\tilde{\phi}_2,
$$
we have
\begin{equation}
\label{eq-18}
u_t+\tilde{u}_1\cdot\nabla u+\tilde{u}\cdot\nabla u_2+\nabla\phi-\Delta u=-\Delta\phi_1\nabla\phi-\Delta\phi\nabla\phi_2,
\end{equation}
and
\begin{equation}
\label{eq-19}
\phi_t+\tilde{u}_1\cdot\nabla\phi+\tilde{u}\cdot\nabla\phi_2+\Delta^2\phi-\kappa\Delta\phi=\Delta[(\phi_1+\omega_0)((\tilde{\phi}_1+\omega_0)^2-2)-
(\phi_2+\omega_0)((\tilde{\phi}_2+\omega_0)^2-2)].
\end{equation}
Multiplying (\ref{eq-18}) by $u$, multiplying (\ref{eq-19}) by $\phi$, integrating over $\mathbb{R}^3$ and using the Gronwall's inequality, taking $T$ small enough, we arrive at (\ref{eq-9}).

\end{proof}

Next, we give the proof of Theorem \ref{thm1.0}.
\begin{proof}[Proof of Theorem \ref{thm1.0}]
By Lemmas \ref{lema-1}, \ref{lem-3}    and a variant of the Banach fixed point theorem, using weak compactness, we complete the proof.
\end{proof}

\section{Small initial data global well-posedness}
In this section, suppose that the condition (\ref{1-2}) holds, we prove the small initial data global well-posedness of solutions for system (\ref{1-1}).
\begin{proof}[Proof of Theorem \ref{thm1.1}]

Taking $\Lambda^{\frac12}$ to (\ref{1-1})$_1$ and (\ref{1-1})$_3$, Taking  $\Lambda^{\frac32}$  to (\ref{1-1})$_3$, multiplying by $\Lambda^{\frac12}u$, $\Lambda^{\frac12}\phi$ and $\Lambda^{\frac32}\phi$ respectively, integrating over $\mathbb{R}^3$, summing them up, we arrive at
\begin{equation}
\label{9-2}\begin{aligned}&
\frac12\frac d{dt}(\|\Lambda^{\frac12}u\|_{L^2}^2+\|\Lambda^{\frac12}\phi\|_{L^2}^2+\|\Lambda^{\frac32}\phi\|_{L^2}^2)
\\&+\|\Lambda^{\frac32}u\|_{L^2}^2+\|\Lambda^{\frac32}\phi\|_{L^2}^2+2\|\Lambda^{\frac52}\phi\|^2_{L^2}+\|\Lambda^{\frac52}\phi\|^2_{L^2}
\\
=&-\int_{\mathbb{R}^3}\Lambda^{\frac12}(u\cdot\nabla u)\cdot\Lambda^{\frac12}udx-\int_{\mathbb{R}^3}\Lambda^{\frac12}(\Delta\phi\cdot\nabla\phi)\cdot\Lambda^{\frac12}udx\\&+\int_{\mathbb{R}^3}\Lambda^{\frac1
2}(u\cdot\nabla\phi)\cdot\Lambda^{\frac12}\phi dx
 +\int_{\mathbb{R}^3}\Lambda^{\frac12}[(\phi+\omega_0)((\phi+\omega_0)^2-2)]\cdot\Lambda^{\frac12}\Delta\phi dx\\
&+\int_{\mathbb{R}^3}\Lambda^{\frac32}(u\cdot\nabla\phi)\cdot\Lambda^{\frac32}\phi dx
+\int_{\mathbb{R}^3}\Lambda^{\frac32}[(\phi+\omega_0)((\phi+\omega_0)^2-2)]\cdot\Lambda^{\frac32}\Delta\phi dx
\\
:=&I_1+I_2+I_3+I_4+I_5+I_6.
\end{aligned}
\end{equation}
Applying  Lemmas \ref{lem2.1} and \ref{KP}, we have
\begin{equation}
\label{9-2c}
\begin{aligned}
I_1\lesssim&\|\Lambda^{\frac12}u\|_{L^6}\|\Lambda^{\frac12}(u\cdot\nabla u)\|_{L^{\frac65}}
\\
\lesssim&\|\Lambda^\frac12u\|_{L^6}(\|\Lambda^{\frac12}u\|_{L^2}\|\nabla u\|_{L^3}+\|u\|_{L^3}\|\Lambda^{\frac32}u\|_{L^2})
\\
\leq&C\|\Lambda^{\frac12}u\|_{L^2}\|\Lambda^{\frac32}u\|_{L^2}^2.
\end{aligned}\end{equation}
Moreover, by using Sobolev's embedding theorem, Lemmas \ref{lem2.1} and \ref{KP}, we obtain
\begin{equation}\label{9-3}
\begin{aligned}
I_2\lesssim&\|\Lambda^{\frac12}u\|_{L^6}\|\Lambda^{\frac12}(\Delta\phi\cdot\nabla\phi)\|_{L^{\frac65}}
\\
\lesssim&\|\Lambda^{\frac12}u\|_{L^6}(\|\Lambda^{\frac52}\phi\|_{L^2}\|\nabla\phi\|_{L^3}+\|\Delta\phi\|_{L^3}\|\Lambda^{\frac32}\phi\|_{L^2})
\\\lesssim&\|\Lambda^{\frac32}u\|_{L^2}(\|\Lambda^{\frac52}\phi\|_{L^2}\|\Lambda^{\frac32}\phi\|_{L^2}+\|\Lambda^{\frac52}\phi\|_{L^2}\|\Lambda^{\frac32}\phi\|_{L^2})
\\
\lesssim&\|\Lambda^{\frac32}\phi\|_{L^2}(\|\Lambda^{\frac32}u\|^2_{L^2}+\|\Lambda^{\frac52}\phi\|^2_{L^2}).
\end{aligned}
\end{equation}
The term $I_3$ satisfies
\begin{equation}\label{9-4}\begin{aligned}
I_3\lesssim&\|\Lambda^{\frac12}\phi\|_{L^6}\|\Lambda^{\frac12}(u\cdot\nabla\phi)\|_{L^{\frac65}}
\\
\lesssim&\|\Lambda^{\frac12}\phi\|_{L^6}( \|\Lambda^{\frac12}u\|_{L^2}\|\nabla\phi\|_{L^3}+ \|u\|_{L^3}\|\Lambda^{\frac32}\phi\|_{L^2})
\\
\lesssim&\|\Lambda^{\frac32}\phi\|_{L^2}( \|\Lambda^{\frac12}u\|_{L^2}\|\Lambda^{\frac32}\phi\|_{L^2}+ \|\Lambda^{\frac12}u\|_{L^2}\|\Lambda^{\frac32}\phi\|_{L^2})
\\
\lesssim& \|\Lambda^{\frac12}u\|_{L^2} \|\Lambda^{\frac32}\phi\|_{L^2}^2.
\end{aligned}
\end{equation}
Similarly, we estimate $I_5$ as
\begin{equation}\label{9-4c}\begin{aligned}
I_5\lesssim&\|\Lambda^{\frac32}\phi\|_{L^6}\|\Lambda^{\frac32}(u\cdot\nabla\phi)\|_{L^{\frac65}}
\\
\lesssim&\|\Lambda^{\frac32}\phi\|_{L^6}( \|\Lambda^{\frac32}u\|_{L^2}\|\nabla\phi\|_{L^3}+ \|u\|_{L^3}\|\Lambda^{\frac52}\phi\|_{L^2})
\\
\lesssim&(\|\nabla\phi\|_{L^3}+\|u\|_{L^3})(\|\Lambda^{\frac52}\phi\|^2_{L^2}+\|\Lambda^{\frac32}u\|^2_{L^2})
\\
\lesssim&(\|\Lambda^{\frac32}\phi\|_{L^2}+\|\Lambda^{\frac12}u\|_{L^2})(\|\Lambda^{\frac52}\phi\|^2_{L^2}+\|\Lambda^{\frac32}u\|^2_{L^2}).
\end{aligned}
\end{equation}
We also have
\begin{equation}
\begin{aligned}\label{9-5c}
I_4\lesssim&\|\Lambda^{\frac52}\phi\|_{L^2}\|\Lambda^{\frac12}[(\phi+\omega_0)((\phi+\omega_0)^2-2)]\|_{L^2}
\\
\lesssim&\|\Lambda^{\frac52}\phi\|_{L^2}(\|\Lambda^{\frac12}(\phi+\omega_0)\|_{L^6}\|\phi+\omega_0-\sqrt{2}\|_{L^6}\|\phi+\omega_0+\sqrt{2}\|_{L^6}
\\&+\|\Lambda^{\frac12}(\phi+\omega_0-\sqrt{2})\|_{L^6}\|\phi+\omega_0 \|_{L^6}\|\phi+\omega_0+\sqrt{2}\|_{L^6}\\&
+\|\Lambda^{\frac12}(\phi+\omega_0+\sqrt{2})\|_{L^6}\|\phi+\omega_0-\sqrt{2}\|_{L^6}\|\phi+\omega_0 \|_{L^6})
\\
\lesssim&\|\Lambda^{\frac52}\phi\|_{L^2}\|\Lambda^{\frac32}\phi\|_{L^2}\|\nabla\phi\|_{L^2}^2
\\\lesssim&\|\Lambda^{\frac52}\phi\|_{L^2}\|\Lambda^{\frac32}\phi\|_{L^2}  \|\Lambda^{\frac12}\phi\|_{L^2} \|\Lambda^{\frac32}\phi\|_{L^2}
\\
\leq&(\|\Lambda^{\frac12}\phi\|_{L^2}^2+ \|\Lambda^{\frac32}\phi\|_{L^2}  ^2 )( \|\Lambda^{\frac52}\phi\|_{L^2}^2+ \|\Lambda^{\frac32}\phi\|_{L^2}^2),
\end{aligned}
\end{equation}
and
\begin{equation}
\begin{aligned}\label{9-5}
I_6\lesssim&\|\Lambda^{\frac72}\phi\|_{L^2}\|\Lambda^{\frac12}[(\phi+\omega_0)^3-2(\phi+\omega_0)\phi]\|_{L^2}
\\
\lesssim&\|\Lambda^{\frac72}\phi\|_{L^2}(\|\Lambda^{\frac32}(\phi+\omega_0)\|_{L^6}\|\phi+\omega_0-\sqrt{2}\|_{L^6}\|\phi+\omega_0+\sqrt{2}\|_{L^6}
\\&+\|\Lambda^{\frac32}(\phi+\omega_0-\sqrt{2})\|_{L^6}\|\phi+\omega_0 \|_{L^6}\|\phi+\omega_0+\sqrt{2}\|_{L^6}\\&
+\|\Lambda^{\frac32}(\phi+\omega_0+\sqrt{2})\|_{L^6}\|\phi+\omega_0-\sqrt{2}\|_{L^6}\|\phi+\omega_0 \|_{L^6})
\\
\lesssim&\|\Lambda^{\frac72}\phi\|_{L^2}\|\Lambda^{\frac52}\phi\|_{L^2}\|\nabla\phi\|_{L^2}^2
\\
\leq&(\|\Lambda^{\frac12}\phi\|_{L^2}^2+ \|\Lambda^{\frac32}\phi\|_{L^2}  ^2 )(\|\Lambda^{\frac72}\phi\|_{L^2}^2+\|\Lambda^{\frac52}\phi\|_{L^2}^2).
\end{aligned}
\end{equation}
Adding (\ref{9-2})-(\ref{9-5}) together gives
\begin{equation}
\label{9-6}\begin{aligned}&
 \frac d{dt}X+Y\leq (X+X^2)Y .
\end{aligned}
\end{equation}
where
$$
X=\|\Lambda^{\frac12}u\|_{L^2}^2+\|\Lambda^{\frac12}\phi\|_{L^2}^2+\|\Lambda^{\frac32}\phi\|_{L^2}^2,
$$
and
$$
Y=\|\Lambda^{\frac32}u\|_{L^2}^2+\|\Lambda^{\frac32}\phi\|_{L^2}^2+2\|\Lambda^{\frac52}\phi\|^2_{L^2}+\|\Lambda^{\frac52}\phi\|^2_{L^2}.
$$
 So, if $X_0$ is sufficiently small, then for any $T\in(0,\infty)$, we have
\begin{equation}
\begin{aligned}\label{ad-3}&
\frac d{dt}X+Y\leq0,
\end{aligned}\end{equation}
which implies
$$
u,\phi,\nabla\phi\in L^{\infty}(0,T;H^{\frac12}),\quad \nabla u,\nabla\phi,\Delta\phi,\nabla\Delta\phi\in L^2(0,T;H^{\frac32}) .
$$

Now, we consider the higher order derivative norm estimates for the solution of system (\ref{1-1}).

Applying $\Lambda^k$ to (\ref{1-1})$_1$ and (\ref{1-1})$_3$, $\Lambda^{k+1}$ to (\ref{1-1})$_3$,  multiplying the resulting identity by $\Lambda^ku$, $\Lambda^k\phi$ and $\Lambda^{k+1}\phi$, respectively, and then integrating over $\mathbb{R}^3$ by parts, summing them up, we arrive at
\begin{equation}
\begin{aligned}
\label{9-7}&
\frac12\frac d{dt}(\|\Lambda^ku\|_{L^2}^2+\|\Lambda^{k}\phi\|_{L^2}^2+\|\Lambda^{k+1}\phi\|_{L^2}^2)
\\
&+\|\Lambda^{k+1}u\|_{L^2}^2+\|\Lambda^{k+1}\phi\|_{L^2}^2+2 \|\Lambda^{k+2}\phi\|_{L^2}^2+\|\Lambda^{k+3}\phi\|_{L^2}^2
\\
=&-\int_{\mathbb{R}^3}\Lambda^k(u\cdot\nabla u)\cdot\Lambda^kudx-\int_{\mathbb{R}^3}\Lambda^k(\Delta\phi\cdot\nabla\phi) \cdot\Lambda^kudx
\\
&+\int_{\mathbb{R}^3}\Lambda^{k   }(u\cdot\nabla\phi)\cdot\Lambda^{k  }\phi dx+\int_{\mathbb{R}^3}\Lambda^{k  }\Delta((\phi+\omega_0)^3-2(\phi+\omega_0))\cdot\Lambda^{k  }\phi dx\\
&+\int_{\mathbb{R}^3}\Lambda^{k +1 }(u\cdot\nabla\phi)\cdot\Lambda^{k+1 }\phi dx+\int_{\mathbb{R}^3}\Lambda^{k+1 }\Delta((\phi+\omega_0)^3-2(\phi+\omega_0))\cdot\Lambda^{k+1 }\phi dx
\\
=&I_7+I_8+I_9+I_{10}+I_{11}+I_{12}.
\end{aligned}
\end{equation}
We will estimate $I_7$-$I_{10}$ one by one in the following.
\begin{equation}\label{9-8}
\begin{aligned}
I_7\lesssim&\|\Lambda^ku\|_{L^6}\|\Lambda^{k}\nabla\cdot(u\otimes u)\|_{L^{\frac65}}
\lesssim \|\Lambda^{k+1}u\|_{L^2}\|\Lambda^{k+1}u\|_{L^2}\|  u\|_{L^3}
\\
\lesssim&\|u\|_{\dot{H}^{\frac12}}\|\Lambda^{k+1}u\|_{L^2}^2
\lesssim(\|u_0\|_{\dot{H}^{\frac12}}+\|\nabla\phi_0\|_{\dot{H}^{\frac12}})\|\Lambda^{k+1}u\|_{L^2}^2
\\
\lesssim&\delta_0\|\Lambda^{k+1}u\|_{L^2}^2,\end{aligned}
\end{equation}
\begin{equation}\label{9-9}
\begin{aligned}
I_8\lesssim&\|\Lambda^ku\|_{L^6}\|\Lambda^{k}(\Delta\phi\cdot\nabla\phi)\|_{L^{\frac65}}
\\
\lesssim&\|\Lambda^{k+1}u\|_{L^2}(\|\Lambda^{k+2}\phi\|_{L^2}\| \nabla\phi\|_{L^3}+\|\nabla^{k+1}\phi\|_{L^3}\|\Delta\phi\|_{L^2})\\
\lesssim&\|\Lambda^{k+1}u\|_{L^2}\left[\|\Lambda^{k+2}\phi\|_{L^2}\| \Lambda^{\frac32}\phi\|_{L^3}\right.\\&\left.+\left(\|\nabla^{k+2}\phi\|_{L^2}^{\frac k{k+\frac12}}\|\Lambda^{\frac32}\phi\|_{L^2}^{\frac{\frac12}{k+\frac12}}\right)\left(\|\nabla^{k+2}\phi\|_{L^2}^{\frac {\frac12}{k+\frac12}}\|\Lambda^{\frac32}\phi\|_{L^2}^{\frac{k}{k+\frac12}}\right)\right]
\\
\lesssim&\|\nabla\phi\|_{\dot{H}^{\frac12}}(\|\Lambda^{k+1}u\|_{L^2}^2+\|\Lambda^{k+2}\phi\|_{L^2}^2)
\\\lesssim&(\|u_0\|_{\dot{H}^{\frac12}}+\|\nabla\phi_0\|_{\dot{H}^{\frac12}})(\|\Lambda^{k+1}u\|_{L^2}^2+\|\Lambda^{k+2}\phi\|_{L^2}^2)
\\
\lesssim&\delta_0(\|\Lambda^{k+1}u\|_{L^2}^2+\|\Lambda^{k+2}\phi\|_{L^2}^2),\end{aligned}
\end{equation}
\begin{equation}
\label{9-10c}
\begin{aligned}
I_9\lesssim&\|\Lambda^{k}\phi\|_{L^6}\|\Lambda^{k}(u\cdot\nabla\phi)\|_{L^{\frac65}}
\\
\lesssim&\|\Lambda^{k}\phi\|_{L^6}(\|\Lambda^{k}u\|_{L^2}\|\nabla\phi\|_{L^3}+\|u\|_{L^3}\|\Lambda^{k+1}\phi\|_{L^2})
\\
\lesssim&\|\Lambda^{k+1}\phi\|_{L^2}\left[\left(\|\Lambda^{k+1}u\|_{L^2}^{\frac{k-\frac12}{k+\frac12}}\|\Lambda^{\frac12}u\|_{L^2}^{\frac1{k+\frac12}}\right)
\left(\|\Lambda^{\frac12}\phi\|_{L^2}^{\frac{k-\frac12}{k+\frac12}}\|\Lambda^{k+1}\phi\|_{L^2}^{\frac1{k+\frac12}}\right)\right.\\&\left.+\|\Lambda^{\frac12}
u\|_{L^2}\|\Lambda^{k+1}\phi\|_{L^2}\right]
\\
\lesssim&(\|\Lambda^{\frac12}u\|_{L^2}+\|\Lambda^{\frac12}\phi\|_{L^2})(\|\Lambda^{k+1}u\|_{L^2}^2+\|\Lambda^{k+1}\phi\|_{L^2}^2)
\\
\lesssim&\delta_0(\|\Lambda^{k+1}u\|_{L^2}^2+\|\Lambda^{k+1}\phi\|_{L^2}^2),\end{aligned}
\end{equation}
\begin{equation}
\label{9-10}
\begin{aligned}
I_{11}\lesssim&\|\Lambda^{k+1}\phi\|_{L^6}\|\Lambda^{k+1}(u\cdot\nabla\phi)\|_{L^{\frac65}}
\\
\lesssim&\|\Lambda^{k+1}\phi\|_{L^6}(\|\Lambda^{k+1}u\|_{L^2}\|\nabla\phi\|_{L^3}+\|u\|_{L^3}\|\Lambda^{k+2}\phi\|_{L^2})
\\
\lesssim&(\|\nabla\phi\|_{\dot{H}^{\frac12}}+\|u\|_{\dot{H}^{\frac12}})(\|\Lambda^{k+1}u\|_{L^2}^2+\|\Lambda^{k+2}\phi\|_{L^2}^2)
\\
\\
\lesssim&\delta_0(\|\Lambda^{k+1}u\|_{L^2}^2+\|\Lambda^{k+2}\phi\|_{L^2}^2),\end{aligned}
\end{equation}
\begin{equation}
\label{9-11c}
\begin{aligned}
I_{10}\lesssim&\|\Lambda^{k+1}\phi\|_{L^2}\|\Lambda^{k+1}((\phi+\omega_0)^3-2(\phi+\omega_0))\|_{L^2}
\\
\lesssim&\|\Lambda^{k+1}\phi\|_{L^2}(\|\Lambda^{k+1}(\phi+\omega_0)\|_{L^6}\|\phi+\omega_0-\sqrt{2}\|_{L^6}\|\phi+\omega_0+\sqrt{2}\|_{L^6}
\\&+\|\Lambda^{k+1}(\phi+\omega_0-\sqrt{2})\|_{L^6}\|\phi+\omega_0 \|_{L^6}\|\phi+\omega_0+\sqrt{2}\|_{L^6}\\&
+\|\Lambda^{k+1}(\phi+\omega_0+\sqrt{2})\|_{L^6}\|\phi+\omega_0-\sqrt{2}\|_{L^6}\|\phi+\omega_0 \|_{L^6})
\\
\lesssim&\|\Lambda^{k+1}\phi\|_{L^2}\|\Lambda^{k+2}\phi\|_{L^2}\|\nabla\phi\|_{L^2}^2\\
\lesssim&\|\Lambda^{k+1}\phi\|_{L^2}\|\Lambda^{k+2}\phi\|_{L^2}(\|\Lambda^{\frac12}\phi\|_{L^2}^2+\|\Lambda^{\frac32}\phi\|^2_{L^2})
\\
\lesssim&(\|\Lambda^{\frac12}\phi\|_{L^2}^2+\|\Lambda^{\frac32}\phi\|^2_{L^2})(\|\Lambda^{k+1}\phi\|_{L^2}^2+\|\Lambda^{k+2}\phi\|_{L^2}^2)
\\
\lesssim&\delta_0^2(\|\Lambda^{k+1}\phi\|_{L^2}^2+\|\Lambda^{k+2}\phi\|_{L^2}^2),
\end{aligned}
\end{equation}\begin{equation}
\label{9-11}
\begin{aligned}
I_{12}\lesssim&\|\Lambda^{k+3}\phi\|_{L^2}\|\Lambda^{k+1}((\phi+\omega_0)^3-2(\phi+\omega_0))\|_{L^2}
\\
\lesssim&\|\Lambda^{k+3}\phi\|_{L^2}(\|\Lambda^{k+1}(\phi+\omega_0)\|_{L^6}\|\phi+\omega_0-\sqrt{2}\|_{L^6}\|\phi+\omega_0+\sqrt{2}\|_{L^6}
\\&+\|\Lambda^{k+1}(\phi+\omega_0-\sqrt{2})\|_{L^6}\|\phi+\omega_0 \|_{L^6}\|\phi+\omega_0+\sqrt{2}\|_{L^6}\\&
+\|\Lambda^{k+1}(\phi+\omega_0+\sqrt{2})\|_{L^6}\|\phi+\omega_0-\sqrt{2}\|_{L^6}\|\phi+\omega_0 \|_{L^6})
\\
\lesssim&\|\Lambda^{k+3}\phi\|_{L^2}\|\Lambda^{k+2}\phi\|_{L^2}\|\nabla\phi\|_{L^2}^2\\
\lesssim&\|\Lambda^{k+3}\phi\|_{L^2}\|\Lambda^{k+2}\phi\|_{L^2}(\|\Lambda^{\frac12}\phi\|_{L^2}^2+\|\Lambda^{\frac32}\phi\|^2_{L^2})
\\
\lesssim&(\|\Lambda^{\frac12}\phi\|_{L^2}^2+\|\Lambda^{\frac32}\phi\|^2_{L^2})(\|\Lambda^{k+3}\phi\|_{L^2}^2+\|\Lambda^{k+2}\phi\|_{L^2}^2)
\\
\lesssim&\delta_0^2(\|\Lambda^{k+3}\phi\|_{L^2}^2+\|\Lambda^{k+2}\phi\|_{L^2}^2),
\end{aligned}
\end{equation}
  It then follows from (\ref{9-7})-(\ref{9-11}) that
\begin{equation}
\label{9-12}\begin{aligned}&
\frac12\frac d{dt}(\|\Lambda^{k}u\|_{L^2}^2+\|\Lambda^{k}\phi\|_{L^2}^2+\|\Lambda^{k+1}\phi\|_{L^2}^2)\\&
 +  \|\Lambda^{k+1}u\|_{L^2}^2+ \|\Lambda^{k+1}\phi\|_{L^2}^2+2\|\Lambda^{k+2}\phi\|^2_{L^2}+ \|\Lambda^{k+3}\phi\|_{L^2}^2
\\
\leq&(4\delta_0+2\delta_0^2) (\|\Lambda^{k}u\|_{L^2}^2+\|\Lambda^{k}\phi\|_{L^2}^2+\|\Lambda^{k+1}\phi\|_{L^2}^2).
\end{aligned}\end{equation}
Choose $\delta_0$ sufficiently small, such that $4\delta_0+2\delta_0^2\leq\frac12$, we obtain
\begin{equation}\begin{aligned}
\label{9-13}&
\|\Lambda^{k}u\|_{L^2}^2+\|\Lambda^{k}\phi\|_{L^2}^2+\|\Lambda^{k+1}\phi\|_{L^2}^2\\&+\int_0^t\left[ \|\Lambda^{k+1}u\|_{L^2}^2+ \|\Lambda^{k+1}\phi\|_{L^2}^2+2\|\Lambda^{k+2}\phi\|^2_{L^2}+ \|\Lambda^{k+3}\phi\|_{L^2}^2\right]ds\\\leq & \|\Lambda^{k}u_0\|_{L^2}^2+\|\Lambda^{k+1}\phi_0\|_{L^2}^2,\quad\hbox{for}~k=1,2,\cdots,N.
\end{aligned}\end{equation}
Hence, we complete the proof of Theorem \ref{thm1.1}.

\end{proof}
\section{Decay rate}
In this section, we consider the decay rate of strong solutions for system (\ref{1-1}). First of all, we derive the evolution of the negative Sobolev norms of the solution to the Cauchy problem (\ref{1-1}).
In order to estimate the convective term and the double-well potential, we shall restrict ourselves to that $s\in[0,\frac12]$.

For the homogeneous Sobolev space, the following lemma holds:
\begin{lemma}
\label{lem5.1}
Suppose that all the assumptions in Theorem \ref{thm1.1} are in force. For $s\in[0,\frac12]$, we have
\begin{equation}
\label{5-1}\begin{aligned}&
\frac d{dt}(\|u(t)\|_{\dot{H}^{-s}}^2 +\| \phi(t)\|_{\dot{H}^{-s}}^2+\|\nabla\phi(t)\|_{\dot{H}^{-s}}^2)\\&
 +\|\nabla u(t)\|^2_{\dot{H}^{-s}}+\|\nabla \phi(t)\|^2_{\dot{H}^{-s}}+ 2\|\Lambda^2 \phi(t)\|^2_{\dot{H}^{-s}}+\|\Lambda^3 \phi(t)\|^2_{\dot{H}^{-s}}  \\
&\leq E(t)(\|u(t)\|_{\dot{H}^{-s}}+\|\phi(t)\|_{\dot{H}^{-s}}+\|\nabla\phi(t)\|_{\dot{H}^{-s}}),
\end{aligned}
\end{equation}where
\begin{equation}
\nonumber\begin{aligned}
E(t)= \|\nabla u\|_{L^2}^2+\|\Lambda^2u\|_{L^2}^2 +\|\nabla ^2\phi\|_{L^2}^2 +\|\Lambda^3\phi\|_{L^2}^2 .
\end{aligned}
\end{equation}
\end{lemma}
\begin{proof}
Applying $\Lambda^{-s}$ to (\ref{1-1})$_1$, applying $\Lambda^{-s}\nabla$ to (\ref{1-1})$_3$,  multiplying the resulting identities by $\Lambda^{-s}u$ and $\Lambda^{-s}\nabla\phi$, and then integrating over $\mathbb{R}^3$ by parts, summing them up, we deduce that
\begin{equation}
\label{5-2}
\begin{aligned}&
\frac12\frac d{dt}\int_{\mathbb{R}^3}(|\Lambda^{-s}u|^2+|\Lambda^{-s}\nabla\phi|^2)dx+\|\Lambda^{-s}\nabla u\|_{L^2}^2+\|\Lambda^{-s+3} \phi\|_{L^2}^2+\|\Lambda^{-s+2}\phi\|_{L^2}^2
\\
=&-\int_{\mathbb{R}^3}\Lambda^{-s}(u\cdot\nabla u)\cdot\Lambda^{-s}udx-\int_{\mathbb{R}^3}\Lambda^{-s}[\Delta\phi\cdot\nabla\phi ]\cdot\Lambda^{-s}udx
\\
&-\int_{\mathbb{R}^3}\Lambda^{-s} (u\cdot\nabla \phi)\cdot\Lambda^{-s} \phi dx+\int_{\mathbb{R}^3}\Lambda^{-s}\Lambda^2((\phi+\omega_0)^3-2(\phi+\omega_0))\cdot\Lambda^{-s} \phi dx\\
&-\int_{\mathbb{R}^3}\Lambda^{-s}\nabla(u\cdot\nabla \phi)\cdot\Lambda^{-s}\nabla\phi dx+\int_{\mathbb{R}^3}\Lambda^{-s}\Lambda^3((\phi+\omega_0)^3-2(\phi+\omega_0))\cdot\Lambda^{-s}\nabla\phi dx
\\
=:&K_1+K_2+K_3+K_4+K_5+K_6.
\end{aligned}
\end{equation}
If $s\in[0,\frac12]$, then $\frac12+\frac s3<1$ and $\frac 3s\geq6$. By using the estimate (\ref{2-4}) of Riesz potential in Lemma \ref{lem2.3}, together with H\"{o}lder's inequality and Young's inequality, we have
\begin{equation}
\begin{aligned}\label{5-3}
K_1=&\int_{\mathbb{R}^3}\Lambda^{-s}(u\cdot\nabla u)\cdot\Lambda^{-s}udx\\
\leq&\|\Lambda^{-s}(u\cdot\nabla u)\|_{L^2}\|\Lambda^{-s}u\|_{L^2}
\\
\lesssim&\|u\cdot\nabla u\|_{L^{\frac1{\frac12+\frac s3}}}\|\Lambda^{-s}u\|_{L^2}
\\
\lesssim&\|u\|_{L^{\frac 3s}}\|\nabla u\|_{L^2}\|\Lambda^{-s}u\|_{L^2}
\\
\lesssim&\|\nabla u\|_{L^2}^{\frac 12+s}\|\Lambda^2u\|_{L^2}^{\frac12-s}\|\nabla u\|_{L^2}\|\Lambda^{-s}u\|_{L^2}
\\
\lesssim&\|\Lambda^{-s}u\|_{L^2}(\|\nabla u\|_{L^2}^2+\|\Lambda^2u\|_{L^2}^2)
,
\end{aligned}
\end{equation}
\begin{equation}
\begin{aligned}\label{5-4}
K_2=&\int_{\mathbb{R}^3} \Lambda^{-s}(\Delta\phi\cdot\nabla\phi) \cdot\Lambda^{-s}udx\\
\leq&\|\Lambda^{-s}(\Delta\phi\cdot\nabla\phi)\|_{L^2}\|\Lambda^{-s}u\|_{L^2}
\\
\lesssim&\|\Delta\phi\cdot\nabla\phi\|_{L^{\frac1{\frac12+\frac s3}}}\|\Lambda^{-s}u\|_{L^2}
\\
\lesssim&\|\nabla\phi\|_{L^{\frac 3s}}\|\Lambda^2\phi\|_{L^2}\|\Lambda^{-s}u\|_{L^2}
\\
\lesssim&\|\Lambda^2\phi\|_{L^2}^{\frac 12+s}\|\Lambda^2\nabla\phi\|_{L^2}^{\frac12-s}\|\Lambda^2\phi\|_{L^2}\|\Lambda^{-s}u\|_{L^2}
\\
\lesssim&\|\Lambda^{-s}u\|_{L^2}(\|\Lambda ^2\phi\|_{L^2}^2+\|\Lambda^3\phi\|_{L^2}^2)
,
\end{aligned}
\end{equation}
\begin{equation}
\label{5-12c}
\begin{aligned}
K_3=&-\int_{\mathbb{R}^3}\Lambda^{-s} (u\cdot\nabla \phi)\cdot\Lambda^{-s} \phi dx\\
\leq&\|\Lambda^{-s} (u\cdot\nabla \phi)\|_{L^2}\|\Lambda^{-s} \phi\|_{L^2}
\\
\lesssim&\| u\cdot\nabla \phi \|_{L^{\frac1{\frac12+\frac s3}}}\|\Lambda^{-s} \phi\|_{L^2}
 \\
\lesssim& \|\nabla\phi\|_{L^2}\|u\|_{L^{\frac 3s}}\|\Lambda^{-s} \phi\|_{L^2}
\\
\lesssim& \|\nabla\phi\|_{L^2}\|\nabla u\|_{L^2}^{\frac 12+s}\|\Delta\phi\|_{L^2}^{\frac12-s} \|\Lambda^{-s} \phi\|_{L^2}
\\
\lesssim&\|\Lambda^{-s} \phi\|_{L^2}(\|\nabla u\|_{L^2}^2+\|\nabla\phi\|_{L^2}^2+ \|\Lambda^2\phi\|_{L^2}^2),
\end{aligned}
\end{equation}
and
\begin{equation}
\label{5-12}
\begin{aligned}
K_5=&-\int_{\mathbb{R}^3}\Lambda^{-s}\nabla(u\cdot\nabla \phi)\cdot\Lambda^{-s}\nabla\phi dx\\
\leq&\|\Lambda^{-s}\nabla(u\cdot\nabla \phi)\|_{L^2}\|\Lambda^{-s}\nabla\phi\|_{L^2}
\\
\lesssim&\|\nabla(u\cdot\nabla \phi)\|_{L^{\frac1{\frac12+\frac s3}}}\|\Lambda^{-s}\nabla\phi\|_{L^2}
\\
\lesssim&(\|\nabla u\|_{L^2}\|\nabla\phi\|_{L^{\frac 3s}}+\|u\|_{L^{\frac3s}}\|\Lambda^2\phi\|_{L^2})\|\Lambda^{-s}\nabla\phi\|_{L^2}
\\
\lesssim&(\|\Lambda^2\phi\|_{L^2}^{\frac 12+s}\|\Lambda^3\phi\|_{L^2}^{\frac12-s}\|\nabla u\|_{L^2}+\|\nabla u\|_{L^2}^{\frac 12+s}\|\Lambda^2u\|_{L^2}^{\frac12-s}\|\Lambda^2\phi\|_{L^2})\|\Lambda^{-s}\nabla\phi\|_{L^2}
\\
\lesssim&\|\Lambda^{-s}\nabla\phi\|_{L^2}(\|\nabla u\|_{L^2}^2+\|\Lambda^2u\|_{L^2}^2+\|\Lambda^3 \phi\|_{L^2}^2+\|\Lambda^2\phi\|_{L^2}^2)
.
\end{aligned}
\end{equation}
Moreover, $K_4$ can be estimated as
\begin{equation}
\begin{aligned}
\label{5-13c}
K_4 
=&\int_{\mathbb{R}^3}\Lambda^{-s}\Delta[(\phi+\omega_0)(\phi+\omega_0+\sqrt{2})(\phi+\omega_0-\sqrt{2})]\cdot\Lambda^{-s} \phi dx
\\
\lesssim&\|\Lambda^{-s} \phi\|_{L^2}\|\Lambda^{-s}\Delta[(\phi+\omega_0)(\phi+\omega_0+\sqrt{2})(\phi+\omega_0-\sqrt{2})]\|_{L^2}
\\
\lesssim&\|\Lambda^{-s}\phi\|_{L^2}\| \Delta[(\phi+\omega_0)(\phi+\omega_0+\sqrt{2})(\phi+\omega_0-\sqrt{2})]\|_{L^{\frac1{\frac12+\frac s3}}}
\\
\lesssim& \|\Lambda^{-s} \phi\|_{L^2}\left(\|\Delta(\phi+\omega_0)\|_{L^2}\|\phi+\omega_0+\sqrt{2}\|_{L^{\infty}}\|\phi+\omega_0-\sqrt{2}\|_{L^{\frac3s}} \right.
\\&\left.+\|\Delta(\phi+\omega_0+\sqrt{2})\|_{L^2}\|\phi+\omega_0 \|_{L^{\infty}}\|\phi+\omega_0-\sqrt{2}\|_{L^{\frac3s}}\right.\\&\left.+\|\Delta(\phi+\omega_0-\sqrt{2})\|_{L^2}\|\phi+\omega_0+\sqrt{2}\|_{L^{\infty}}\|\phi+\omega_0 \|_{L^{\frac3s}}   \right)
\\
\lesssim& \|\Lambda^{-s}\phi\|_{L^2}\|\Delta\phi\|_{L^2}\|\nabla\phi\|^{\frac12}_{L^2}\|\Delta\phi\|_{L^2}^{\frac12}\|\nabla\phi\|^{\frac12+s}_{L^2}\|\Delta\phi\|_{L^2}^{\frac12-s}
\\
\lesssim&\|\Lambda^{-s} \phi\|_{L^2}( \|\nabla^2\phi\|_{L^2}^2+\|\nabla\phi\|_{L^2}^2)  ,  .
\end{aligned}
\end{equation}
 We estimate $K_6$ as
\begin{equation}
\begin{aligned}
\label{5-13}
K_6=&\int_{\mathbb{R}^3}\Lambda^{-s}\Lambda^3((\phi+\omega_0)^3-2(\phi+\omega_0))\cdot\Lambda^{-s}\nabla\phi dx
\\
=&\int_{\mathbb{R}^3}\Lambda^{-s}\Lambda^3[(\phi+\omega_0)(\phi+\omega_0+\sqrt{2})(\phi+\omega_0-\sqrt{2})]\cdot\Lambda^{-s}\nabla\phi dx
\\
\lesssim&\|\Lambda^{-s}\nabla\phi\|_{L^2}\|\Lambda^{-s}\Lambda^3[(\phi+\omega_0)(\phi+\omega_0+\sqrt{2})(\phi+\omega_0-\sqrt{2})]\|_{L^2}
\\
\lesssim&\|\Lambda^{-s}\nabla\phi\|_{L^2}\| \Lambda^3[(\phi+\omega_0)(\phi+\omega_0+\sqrt{2})(\phi+\omega_0-\sqrt{2})]\|_{L^{\frac1{\frac12+\frac s3}}}
\\
\lesssim& \|\Lambda^{-s}\nabla\phi\|_{L^2}\left(\|\nabla^3(\phi+\omega_0)\|_{L^2}\|\phi+\omega_0+\sqrt{2}\|_{L^{\infty}}\|\phi+\omega_0-\sqrt{2}\|_{L^{\frac3s}} \right.
\\&\left.+\|\nabla^3(\phi+\omega_0+\sqrt{2})\|_{L^2}\|\phi+\omega_0 \|_{L^{\infty}}\|\phi+\omega_0-\sqrt{2}\|_{L^{\frac3s}}\right.\\&\left.+\|\nabla^3(\phi+\omega_0-\sqrt{2})\|_{L^2}\|\phi+\omega_0+\sqrt{2}\|_{L^{\infty}}\|\phi+\omega_0 \|_{L^{\frac3s}}   \right)
\\
\lesssim& \|\Lambda^{-s}\nabla\phi\|_{L^2}\|\nabla^3\phi\|_{L^2}\|\nabla\phi\|^{\frac12}_{L^2}\|\Delta\phi\|_{L^2}^{\frac12}\|\nabla\phi\|^{\frac12+s}_{L^2}\|\Delta\phi\|_{L^2}^{\frac12-s}
\\
\lesssim&\|\Lambda^{-s}\nabla\phi\|_{L^2}( \|\nabla^2\phi\|_{L^2}^2+\|\nabla^3\phi\|_{L^2}^2+\|\nabla\phi\|_{L^2}^2)  ,
\end{aligned}
\end{equation}where we have used the fact that $\nabla\phi\in L^{\infty}(0,T;H^N)$.
Combining (\ref{5-2})-(\ref{5-13}) together, we obtain (\ref{5-1}) and   the proof is complete.

\end{proof}

In the following, we give the proof of Theorem \ref{thm1.2}.
\begin{proof}[Proof of Theorem \ref{thm1.2}]
 Define
$$
\mathcal{E}_{-s}(t):=\|\Lambda^{-s}u(t)\|_{L^2}^2+ \|\Lambda^{-s} \phi(t)\|_{L^2}^2+ \|\Lambda^{-s}\nabla\phi(t)\|_{L^2}^2.
$$
For inequality (\ref{5-1}), integrating in time, by the bound (\ref{9-13}), we have
\begin{equation}
\label{6-5c}
\begin{aligned}
\mathcal{E}_{-s}(t)\leq&\mathcal{E}_{-s}(0)+C\int_0^tE(t)\sqrt{\mathcal{E}_{-s}(\tau)}d\tau
\\
\leq&C_0\left(1+\sup_{0\leq\tau\leq t}\sqrt{\mathcal{E}_{-s}(\tau)}d\tau\right),
\end{aligned}
\end{equation}
which implies (\ref{1-4}) for $s\in[0,\frac12]$, that is
\begin{equation}
\label{6-5}
\|\Lambda^{-s}u(t)\|_{L^2}^2 +\|\Lambda^{-s} \phi(t)\|_{L^2}^2 +\|\Lambda^{-s}\nabla\phi(t)\|_{L^2}^2\leq C_0.
\end{equation}
Moreover, if  $l=1,2,\cdots,N-1$, we may use Lemma \ref{lem2.2} to have
$$
\|\Lambda^{l+1}f\|_{L^2}\geq C\|\Lambda^{-s}f\|_{L^2}^{-\frac1{l+s}}\|\Lambda^lf\|_{L^2}^{1+\frac1{l+s}}.
$$
Then, by this facts and (\ref{6-5}), we get
\begin{equation}\label{7-1}
\|\Lambda^{l+1}(u ,\phi,\nabla\phi)\|_{L^2}^2\geq C_0(\|\Lambda^{l}(u, \phi,\nabla\phi)\|_{L^2}^2)^{1+\frac1{k+s}}.
\end{equation}
Thus, for $k=0,1,2,\cdots,N-1$, we deduce from (\ref{9-13})  the following inequality
\begin{equation}
\label{7-2}
\frac d{dt}(\|\Lambda^{k}u\|_{L^2}^2+\|\Lambda^{k}\phi\|_{L^2}^2+\|\Lambda^{k+1}\phi\|_{L^2}^2)+C_0\left(\|\Lambda^{k}u\|_{L^2}^+\|\Lambda^{k}\phi\|_{L^2}^2+\|\Lambda^{k+1}\phi\|_{L^2}^2\right)^{1+\frac1{l+s}}\leq 0.
\end{equation}
Solving this inequality directly gives
\begin{equation}
\label{7-3}
\|\Lambda^{k}u\|_{L^2}^2+\|\Lambda^{k}\phi\|_{L^2}^2+\|\Lambda^{k+1}\phi\|_{L^2}^2\leq C_0(1+ t)^{-l-s},\quad\hbox{for}~l=1,2,\cdots,N-1,
\end{equation}
which means (\ref{1-5}) holds. Hence, we   complete the proof of Theorem \ref{thm1.2}.

\end{proof}

\section*{Acknowledgement} This paper was supported by the   Natural Science Foundation of China (grant No. 11401258) and China Postdoctoral Science Foundation (grant No. 2015M581689).   The author also would like to thank Prof. Andrea Giorgini and Prof. Yong Zhou for their careful reading and suggestions.

\section*{Data availability}The data that support the findings of this study are available from the corresponding author upon reasonable request.

}

\end{document}